\documentclass[12pt]{article}
\usepackage{url}
\usepackage{fancyhdr}
\usepackage{setspace}  
\usepackage{mathptmx}
\usepackage{latexsym}
\usepackage{amssymb}
\usepackage{amsthm}
\usepackage{amsmath}
\newtheorem{theorem}{Theorem}
\newtheorem{lemma}{Lemma}
\newtheorem{corollary}{Corollary}

\begin{document}
\title{\bf Geometric realization of toroidal quadrangulations without hidden symmetries}
\author{\Large Serge Lawrencenko \medskip  \\
\small Faculty of Control and Design \\
\small Russian State University of Tourism and Service\\
\small \tt lawrencenko@hotmail.com}
\date {}
\maketitle

\begin{abstract}
It is shown that each quadrangulation of the 2-torus by the Cartesian product of two cycles 
can be geometrically realized in (Euclidean) 4-space without hidden symmetries---that is, so that each combinatorial cellular automorphism of the quadrangulation extends to a geometric symmetry of its Euclidean realization. Such realizations turn out to be new regular toroidal geometric 2-polyhedra which are inscribed in the Clifford 2-torus in 4-space, just as the five regular spherical 2-polyhedra are inscribed in the 2-sphere in 3-space. The following are two open problems: Realize geometrically (1) the regular triangulations and (2) the regular hexagonizations of the 2-torus without hidden symmetries in 4-space.
\end{abstract}

{\bf Keywords:} quadrangulation, torus, Cartesian product of graphs, geometric realization, symmetry group, regular polyhedron.

{\bf MSC Classification:} 51M20 (Primary); 52B15, 51F15, 20F65, 05C25, 57M20, 57M15 (Secondary). 

\section{Introduction}
The concept of hidden symmetry was introduced by Hermann Weyl \cite{w}. The basis of this concept is the understanding that if $P$ is a polyhedron in Euclidean $d\mbox{-}$space $\mathbb{R}^d$ with the group of obvious (that is, Euclidean) symmetries ${\rm{Sym}}(P)$, $P$ may have hidden symmetries which are elements of a larger group---the combinatorial cellular automorphism group ${\rm{Aut}}(P)$. Revealing hidden symmetries in polyhedra is an important problem along with visual geometric realization of those symmetries.

A fundamental theorem by Peter Mani \cite{m} states that any polygonization of the 2-sphere with a 3-connected graph is realizable in $\mathbb{R}^3$ without hidden symmetries---that is, each combinatorial cellular automorphism of the polygonization extends to a geometric symmetry of its Euclidean realization. In this paper this result is generalized for quadrangulations of the 2-torus. 

The ({\it Cartesian}) {\it product} of two graphs (that is, simplicial 1-complexes) $G_1$ and $G_2$ with disjoint vertex sets $V(G_1)$ and $V(G_2)$ is denoted by $G_1 \times G_2$ and is defined to be the graph whose vertex set is $V(G_1 \times G_2)=V(G_1) \times V(G_2)$ and in which two vertices $(u_1, u_2)$ and $(v_1, v_2)$ are connected by an edge if either \{$u_1 = v_1$ and the vertices $u_2,\, v_2$ are connected by an edge in $G_2$\} or \{$u_2 = v_2$ and the vertices $u_1,\, v_1$ are connected by an edge in $G_1$\}. 

Let $C_n$ denote the graph that is a simple cycle of length $n$ ($n \ge 3$). The Cartesian product $C_n \times C_k$ naturally embeds in the 2-torus $\mathbb T^2$ as
$n$ parallels and $k$ meridians (or vice versa) which collectively produce a quadrangulation $C_n \times C_k \hookrightarrow \mathbb T^2$ denoted by $Q_{n,k}$. This quadrangulation has $nk$ vertices,   $2nk$ edges, and $nk$ quadrilateral faces. 

An analogue of Mani's Theorem for toroidal quadrangulations $Q_{n,k}$ is established in Section 3 as part of a more general result---Theorem 1. Geometric realization without hidden symmetries of polygonizations of 
2-manifolds with transitive automorphism groups in $\mathbb{R}^d$ produces regular 2-polyhedra in $\mathbb{R}^d$. An infinite series of such 2-polyhedra is constructed in Section 4, in addition to the noble toroidal hexadecahedron constructed by the author in \cite{l}. (A ``noble polyhedron'', a term introduced by Branko Gr\"unbaum \cite{g}, names a polyhedron whose full symmetry group is vertex- and face-transitive but not necessarily edge-transitive.) More precisely, it is shown in Section 4 that the quadrangulations $Q_{n,k}$ are realizable as noble toroidal 2-polyhedra and, in the specific case $n=k$, even regular toroidal 2-polyhedra which are inscribed in the Clifford 2-torus in $\mathbb{R}^4$, just as the five regular spherical 2-polyhedra are inscribed in the 2-sphere in $\mathbb{R}^3$. It is an open problem to realize geometrically the regular triangulations and hexagonizations without hidden symmetries in 4-space (Section 5).

\section{Concepts, implications, and an example}
Let $G$ be a finite simplicial 1-complex, or in other words, a simple undirected {\it graph}, and let $\mathbb M^2$ be a 2-manifold. The {\it faces} of a topological embedding
$h: G \hookrightarrow \mathbb {M}^2$ are the components of $\mathbb M^2 - h(G)$. Such an embedding is called a {\it polygonization} of $\mathbb {M}^2$ with the graph $G$ provided that the closure of each face is homeomorphic to a closed 2-disc. A polygonization in which each face is bounded by a cycle of $G$ with length 4 is called a ({\it topological}) {\it quadrangulation.} An important family of quadrangulations $Q_{n,k}$ was defined in the Introduction.

On the combinatorial side, a quadrangulation corresponds to an abstract 2-complex (in which each 2-cell corresponds to a quadrilateral face) provided that the intersection of the closures of any two faces is either empty, a vertex or an edge (including its two vertices) of $G$. 

Let $K^p$ and $L^q$ be finite abstract CW-complexes of dimensions $p$ and $q$ ($p \le q$), with vertex sets $V(K^p)$ and $V(L^q)$, respectively. A ({\it combinatorial cellular}) {\it homomorphism} 
$K^p \rightarrow L^q$ is defined to be a {\it cellular mapping} $\mu: V(K^p) \rightarrow V(L^q)$---
that is, if $v_0, v_1, \ldots, v_r$ are the vertices of a cell of $K^p$, then 
$\mu(v_0), \mu(v_1), \ldots,\mu(v_r)$ are the vertices of a cell of $L^q$. An injective homomorphism is called a {\it monomorphism,} and a surjective monomorphism is called an {\it isomorphism.} Especially, an isomorphism $K^p \rightarrow K^p$ is called an {\it automorphism} of $K^p$. The automorphism group of $K^p$ is denoted ${\rm {Aut}} (K^p)$. The symbol $\equiv$ designates identity of groups. 

\begin{lemma}
\begin{equation}
|{\rm{Aut}} (C_n \times C_k)| =
  \begin{cases}
   4nk & {\rm{if}} \ n \ne k \\
   8n^2 & {\rm{if}} \ n=k \ne 4 \\
   384 & {\rm{if}} \ n=k=4
  \end{cases}
\end{equation}
\begin{equation}
|{\rm{Aut}} (Q_{n,k})| =
  \begin{cases}
   4nk & {\rm{if}} \ n \ne k \\
   8n^2 & {\rm{if}} \ n=k
  \end{cases}
\end{equation}
\end{lemma}  

\begin{proof}
Firstly we prove Eq.~(1) of the lemma by standard graph-theoretic techniques \cite{h, hp}. Note that $C_n$ can be extended to a nontrivial product of graphs if and only if $n=4$, in which case $C_4 = I \times J$, where $I$ and $J$ denote two disjoint 1-simplices. In this sense, for $n \ne 4$, $C_n$ is a connected prime graph, and therefore 
(see \cite {hp}) 
$$
{\rm{Aut}} (C_n \times C_n) \equiv {\rm{Aut}}(C_n) \, {\rm{wr}} \, S_2 \equiv D_n \, {\rm{wr}} \, S_2   
$$
which is the wreath product (called the composition in \cite{h}) of the dihedral group $D_n$ by the symmetric group $S_2$ and has order $|D_n|^2 |S_2| = 8n^2$.
Furthermore, for $n \ne k$, $C_n$ and $C_k$ are relatively prime graphs with respect to the graph product operation, and therefore (see \cite{h}) 
$$
{\rm{Aut}} (C_n \times C_k) \equiv {\rm{Aut}}(C_n) \times {\rm{Aut}}(C_k) \equiv D_n \times D_k
$$
which is the direct product of two dihedral groups and has order 
$|D_n||D_k| = 4nk$.
Finally, it is well known \cite{h} that the automorphism group of the graph 
$C_4 \times C_4$ (that is, the 1-skeleton of the 4-cube) has order 384.

We now proceed to prove Eq.~(2) of the lemma. Clearly, the group ${\rm{Aut}} (Q_{n,k})$ is always vertex-transitive. If $n=k$, the stabilizer of each vertex is isomorphic to the dihedral group $D_4$, and therefore 
${|{\rm{Aut}} (Q_{n,n})|} = |V(Q_{n,n})| \times|D_4| = 8n^2$. If $n \ne k$, no cycle with length $n$ can map onto a cycle with length $k$, and therefore the stabilizer of each vertex $v$ of $Q_{n,k}$ is the group of order 4 generated by the permutations $(\alpha, {\rm{id}}_2)$ and $({\rm{id}}_1, \beta)$ of the vertex set $V(Q_{n,k}) = V(C_n \times C_k)$, where ${\rm {id}}_1$ and ${\rm {id}}_2$ are identical but $\alpha$ and $\beta$ are non-identical involutive automorphisms of the factors $C_n$ and $C_k$ (respectively) that fix the vertex $v$. Therefore $|{\rm{Aut}}(Q_{n,k})| = |V(Q_{n,k})| \times 4 = 4nk$.
\end{proof}

A {\it flag} of a 2-complex is defined to be a triple of pairwise incident elements in the form of (vertex, edge, face). 

\begin{corollary}
For any $n, \, k \ge 3$, $Q_{n,k}$ is a noble quadrangulation in the sense that 
${\rm{Aut}}(Q_{n,k})$ is vertex-transitive and face-transitive (but not edge-transitive when $n \ne k$).
Furthermore, for any $n$, $Q_{n,n}$ is a regular quadrangulation in the sense that 
${\rm{Aut}}(Q_{n,n})$ is flag-transitive, which provides maximum possible order of the group. 
\end{corollary}

For a finite abstract $p\mbox{-}$complex $K^p$, let $|K^p|: K^p \hookrightarrow \mathbb {R}^d$ denote a (geometric) realization of $K^p$---that is, a $p\mbox{-}$polyhedron whose cell structure is naturally  inherited from $K^p$ and whose ``corner points'' correspond naturally and bijectively to the vertices of $K^p$. Euclidean motions leaving $|K^p|$ invariant form a finite subgroup of the group of Euclidean motions of $\mathbb R^d$. That subgroup is denoted by ${\rm{Sym}}(|K^p|)$ and is called the 
({\it full}) {\it symmetry group of $|K^p|$.} Acting on the vertex set $V(K^p)$, the group ${\rm{Sym}}(|K^p|)$ corresponds to a subgroup of ${\rm{Aut}}(K^p)$ and is often understood combinatorially in this paper---that is, as the group of the corresponding permutations of $V(K^p)$. It will be clear from the context whether we understand 
${\rm{Sym}}(|K^p|)$ combinatorially or geometrically.

A {\it polytope} is defined to be the convex hull of a finite set of points in $\mathbb {R}^d$. A {\it $d\mbox{-}$polytope} is a $d\mbox{-}$dimensional polytope.
Let $\mu: K^p \rightarrow L^q$ be a combinatorial cellular monomorphism and let a realization 
$|L^q|: L^q \hookrightarrow \mathbb{R}^{q+1}$ be given by the boundary complex of some 
$(q+1)\mbox{-}$polytope in $\mathbb{R}^{q+1}$. 
Denote by $\hat{K}^p$ the image $\mu(K^p)$ and denote by $|\hat{K}^p|$ the realization 
$\hat{K}^p \hookrightarrow \mathbb{R}^{q+1}$ naturally induced by $|L^q|$. Therefore we obtain a realization 
$|\hat{K}^p|: K^p \hookrightarrow \mathbb{R}^{q+1}$ as
\begin{equation}
K^p \rightarrow \hat{K}^p \rightarrow |\hat{K}^p| \subseteq |L^q| \subset \mathbb{R}^{q+1}.
\end{equation}
\noindent Realization (3) of $K^p$ by the polyhedron $|\hat{K}^p|$ in $\mathbb{R}^{q+1}$ is said to be {\it without hidden symmetries} provided that each automorphism of $K^p$ is induced by some Euclidean symmetry of $|\hat{K}^p|$.

On the algebraic side, assuming that the origin of $\mathbb {R}^{q+1}$ is fixed by the whole symmetry group of 
$|\hat{K}^p|$, if realization (3) of $K^p$ by $|\hat{K}^p|$ in $\mathbb{R}^{q+1}$ is without hidden symmetries, then the group ${\rm{Sym}}(|\hat{K}^p|) \subset {\rm{O}}(q+1)$ provides a faithful representation of the group ${\rm{Aut}}(K^p)$ of degree $q+1$.

\begin{lemma}
For the absence of hidden symmetries in realization (3) it is sufficient that the following three conditions hold simultaneously:

{\rm {(i)}}   ${\rm {Sym}} (|L^q|) \subseteq {\rm {Sym}} (|\hat{K}^p|),$

{\rm {(ii)}}  $|V(L^q)| = |V(K^p)|,$

{\rm {(iii)}} $|{\rm{Sym}}(|L^q|)| = |{\rm{Aut}}(K^p)|.$
\end{lemma}

\begin{proof}
By condition (i), each symmetry of $|L^q|$ is a symmetry of $|\hat{K}^p|$ 
and, therefore, induces some automorphism of $K^p$. 
By condition (ii), distinct symmetries of $|L^q|$ induce distinct automorphisms of $K^p$.
By condition (iii), each automorphism of $K^p$ is induced by some symmetry of $|L^q|$ which is also a symmetry of $|\hat{K}^p|$ by (i).
\end{proof}

For example, for $p=q=1$ 
the cycle $K^1 = C_n$ is realized by the boundary complex  
$B(P^2_n)$ of a regular Euclidean $n\mbox{-}$gon $P^2_n$
in $\mathbb{R}^2$. Note that $P^2_n$ is a 2-polytope in $\mathbb{R}^2$, 
and also note that the groups ${\rm{Sym}} (B(P^2_n))$ and ${\rm{Sym}} (P^2_n)$
are identical as permutation groups on the set $V(P^2_n) = V(|\hat{K}^1|).$ 
The complex $B(P^2_n)$ corresponds to $|\hat{K}^1| = |L^1|$ in (3), and therefore condition (ii) of Lemma 2 holds.
Since both groups ${\rm{Sym}} (P^2_n)$ and ${\rm{Aut}}(C_n)$ act on the set $V(C_n)$
as the $n\mbox{-}$gonal dihedral group $D_n$, condition (iii) also holds. Furthermore, since each element of the group
${\rm{Sym}} (P^2_n)$ leaves the 1-skeleton $C_n$ setwise invariant, condition (i) holds too, and by Lemma 2 the graph $C_n$ is realized by $B(P^2_n)$ without hidden symmetries.

\section{Realization of toroidal quadrangulations}

The construction of this section is a generalization of the example with the cycle $C_n$ at the end of the preceding section.

The Cartesian product of a regular $n\mbox{-}$gon $P ^2_n$ and a regular $k\mbox{-}$gon $P ^2_k$ ($n,k \ge 3$) is known \cite{o} as the {\it$n,k\mbox{-}$duoprism}, denoted by $P^4_{n,k}$ in this paper, and is a 4-polytope as the convex hull of the set of $nk$ points  
$M_{ij}\big{(}\cos \frac {2 \pi i} n, \sin \frac {2 \pi i} n, \cos \frac {2 \pi j} k, \sin \frac {2 \pi j} k\big{)}$, where $i=0,1,\ldots, n-1$ and $j=0,1,\ldots, k-1$. Therefore the vertices $M_{ij}$ of $P^4_{n,k}$ lie on the Clifford 2-torus in $\mathbb{R}^4$ and the boundary complex $B(P^4_{n,k})$ is a 3-polyhedron [which plays the role of $|L^3|$ in realization (3)] inscribed in that 2-torus. 
It is not hard to observe that the 1-skeleton of $B(P^4_{n,k})$
is the graph $C_n \times C_k$ [which plays the role of $K^1$ in realization (3)] and that the 2-skeleton of $B(P^4_{n,k})$ contains a realization 
$\hat{Q}_{n,k}$ of $Q_{n,k}$ plus $n$ more $k\mbox{-}$gons and $k$ more 
$n\mbox{-}$gons.

\begin{theorem}
For any $n, \, k \ge 3$, the graph $C_n \times C_k$ is geometrically realized in $\mathbb{R}^4$ without hidden symmetries by the 
1-skeleton of the corresponding duoprism's boundary complex $B(P^4_{n,k})$, and the toroidal quadrangulation $Q_{n,k}$ is geometrically realized without hidden symmetries in the 2-skeleton of $B(P^4_{n,k})$. 
\end{theorem}

\begin{proof}
This proof is divided into three cases.    
 
{\it Case 1. $n \ne k$.} We have
$$
{\rm{Sym}} (P^4_{n,k}) \equiv {\rm{Sym}}(P^2_n \times P^2_k) \equiv 
{\rm{Sym}}(P^2_n) \times {\rm{Sym}}(P^2_k) 
\equiv D_n \times D_k,
$$
The group obtained is the direct product of two dihedral groups and has order $4nk$. Condition (iii) of Lemma 2 holds by Lemma 1(1). Condition (ii) holds trivially. Furthermore, since each element of the group ${\rm{Sym}} (P^4_{n,k})$ leaves the 1-skeleton $C_n \times C_k$ setwise invariant, condition (i) holds as well. Therefore, by Lemma 2, the graph $C_n \times C_k$ is realized by the 1-skeleton of $B(P^4_{n,k})$ without hidden symmetries. Then $Q_{n,k}$ is realized without hidden symmetries in the 2-skeleton of $B(P^4_{n,k})$ because ${\rm{Aut}}(Q_{n,k}) \subseteq {\rm{Aut}}(C_n \times C_k)$.

{\it Case 2. $n=k \ne 4$.} Similarly, here we have
$$
{\rm{Sym}} (P^4_{n,n}) \equiv {\rm{Sym}}(P^2_n \times P^2_n) \equiv 
{\rm{Sym}}(P^2_n) \,{\rm{wr}}\, S_2
\equiv D_n \, {\rm{wr}}\, S_2
$$
which is the wreath product of the dihedral group $D_n$ by the symmetric group $S_2$ and has order 
$|D_n|^2 |S_2| = 8n^2$. By Lemma 1, 
$|{\rm{Sym}}(P^4_{n,n})| = |{\rm{Aut}}(C_n \times C_n)|$, and the proof is completed as in Case  1. 

{\it Case 3. $n=k=4$.} In this case $P^4_{4,4}$ is the 4-cube and ${\rm{Sum}}(P^4_{4,4})$
is the hyperoctahedral group of order 384. Again by Lemma 1,
$|{\rm{Sym}}(P^4_{4,4})| = |{\rm{Aut}}(C_4 \times C_4)|$, and the proof is completed as above.
\end{proof}

\begin{corollary}
${\rm{Aut}}(C_n \times C_k) \equiv {\rm{Sym}}(P^4_{n,k})$ ($n,k \ge 3$), where 
${\rm{Sym}}(P^4_{n,k})$ is regarded as a permutation group on the set $V(C_n \times C_k)$.
\end{corollary}

Notably, Case 3 is the only case where $|{\rm{Aut}}(C_n \times C_k)| \ne |{\rm{Aut}} (Q_{n,k})|$ (compare to Lemma 1), which means that there are at least two copies of $Q_{4,4}$ in the 2-skeleton of the 4-cube (all realized without hidden symmetries 
by Theorem 1). The exact number of such copies is equal to 3, which can be found by \cite[formula (4)]{ckl} along with Lemma 1 as the ratio $|{\rm{Aut}}(C_4 \times C_4)| / |{\rm{Aut}} (Q_{4,4})|= 384/128 =3$.

\section{Regular 2-polyhedra}

Any geometric realization of a polygonization of a closed 2-manifild in $\mathbb{R}^d$
without hidden symmetries is called a {\it regular 2-polyhedron} provided that the automorphism group of that polygonization is flag-transitive, and, following Branko Gr\"unbaum \cite{g}, is called a {\it noble 2-polyhedron} provided that that group is vertex-transitive and face-transitive but not necessarily edge-transitive. A noble 2-polyhedron is both isogonal and isohedral---that is, all polyhedral angles at the vertices are congruent and all the faces are congruent. In addition to the above listed congruencies, a regular 2-polyhedron has all dihedral angles congruent. The following is a corollary of the combination of Corollary 1 and Theorem 1.

\begin{corollary}
For $n, \, k \ge 3$ the quadrangulation $Q_{n,k}$ is realizable as a noble toroidal 2-polyhedron in $\mathbb{R}^4$, and $Q_{n,n}$ is realizable as a regular toroidal 2-polyhedron in $\mathbb{R}^4$.  
\end{corollary}

Note that all regular toroidal 2-polyhedra $|\hat{Q}_{n,n}|$ found in Section 3 are inscribed in the Clifford 2-torus in $\mathbb{R}^4$, just as the five regular spherical 2-polyhedra are inscribed in the 2-sphere 
in $\mathbb{R}^3$.

\section{Open problems}
A classification of all regular toroidal polygonizations was given by Coxeter \cite[pp. 25--27]{c}.
They split into three series:  an infinite series of self-dual quadrangulations (with the degree of each vertex equal to $\delta =4$), and two infinite (dual) series of triangulations ($\delta =6$) and hexagonizations ($\delta =3$). The quadrangulations are realized by regular toroidal 2-polyhedra in $\mathbb{R}^4$ by Corollary 3. The following are two open problems: Realize geometrically (1) the regular triangulations and (2) the regular hexagonizations without hidden symmetries in $\mathbb{R}^4$.

\section{Concluding remarks}
So, Mani's Theorem doesn't extend to polygonizations of 2-manifolds of higher genera unless we increase the dimension of the ambient Euclidean space.  A rectangle subdivided by $n$ vertical and $k$ horizontal lines into $nk$ congruent subrectangles in $\mathbb{R}^2$ gives a realization of the corresponding quadrangulation of the 2-disc without hidden symmetries. Furthermore, we can isometrically bend the subdivided rectangle along the vertical lines and then indentify the left and right sides in $\mathbb{R}^3$ to obtain a realization of the corresponding quadrangulation of the 2-cylinder without hidden symmetries. Finally, we can isometrically bend the so obtained 2-cylinder along the horizontal circles (which progressed from the original horizontal lines) and then identify the upper and lower circles in $\mathbb{R}^4$ to obtain, by Theorem 1, a realization of the corresponding quadrangulation of the 2-torus without hidden symmetries. The first of the constructed quadrangulations is a prismatic 2-polytope, the second sits in the boundary complex of a prismatic 3-polytope, and the third sits in the 2-skeleton of the boundary complex of a prismatic 4-polytope.
\medskip

\bibliographystyle{model1-num-names}
\bibliography{<your-bib-database>}

\end{document}